\documentclass[12pt]{amsart}

\usepackage[utf8]{inputenc} 
\usepackage[T1]{fontenc}
\usepackage{lmodern}
\usepackage{comment}
\usepackage{amsmath}
\usepackage{amssymb}

\newtheorem{theorem}{Theorem}
\newtheorem{corollary}[theorem]{Corollary}
\newtheorem{lemma}[theorem]{Lemma}

\newcommand{\C}{\mathbb C}
\newcommand{\R}{\mathbb R}
\newcommand{\Z}{\mathbb Z}
\newcommand{\N}{\mathbb N}

\begin{document}

\title[Approximation by entire functions]{Approximation of a function and its derivatives by entire functions}

\author{P. M.~Gauthier and J.~Kienzle}

\address{D\'epartement de math\'ematiques et de statistique, Universit\'e de Montr\'eal,
CP-6128 Centreville, Montr\'eal,  H3C3J7, CANADA}
\email{gauthier@dms.umontreal.ca, julie.kinzlie@umontreal.ca}


\thanks{Research of second author supported by NSERC (Canada)}

\begin{abstract}
A simple proof is given for the fact that, for $m$ a non-negative integer, a function $f\in C^{(m)}(\R),$ and an arbitrary positive continuous function $\epsilon,$ there is an entire function $g,$ such that $|g^{(i)}(x)-f^{(i)}(x)|<\epsilon(x),$ for all $x\in\R$ and for each $i=0,1\cdots,m.$ We also consider the situation, where $\R$ is replaced by an open interval. 
\end{abstract}

\maketitle

\section{Introduction}

For an open interval  $I = (a, b), -\infty \leq a  <  b \leq +\infty,$ 
and $m=0,1,\cdots,$  denote by $C^{(m)}(I)$ the space of functions  $f:I\rightarrow\C,$ whose derivitaves $f^{(0)},f^{(1)},\cdots,f^{(m)}$ exist and are continuous on $I.$ For a closed interval $\overline I,$ let $C^{(m)}(\overline I)$ be the space of functions $f\in C^{(m)}(I),$ such that $f^{(0)},f^{(1)},\cdots,f^{(m)}$ extend continuously to $\overline I.$ By abuse of notation, we continue to denote these extensions by $f^{(j)}$ respectively.  The following generlalization of the Weierstrass approximation theorem is well-known. 

\begin{theorem}
For $-\infty<a<b<+\infty,$ and $m$ a non-negative integer, let $f\in C^{(m)}([a,b])$ and $\epsilon>0.$ Then, there is a polynomial $p,$ such that $|p^{(i)}(x)-f^{(i)}(x)|<\epsilon,$ for all $x\in [a,b]$ and $i=0,1,\cdots,m.$  
\end{theorem}

To prove this theorem, we merely approximate $f^{(m)}$ by a polynomial and integrate $m$ times. 

An other  extension of the Weierstrass theorem, not as well known as it should be,  is the following  theorem of Carleman \cite{C}, in which a bounded interval is replaced by the entire real line.  
Denote by  $C^+(X)$ the {\em positive} continuous functions on a set $X.$ 

\begin{theorem} (Carleman) Let $f \in C(\mathbb{R})$ and $\epsilon\in C^+(\mathbb{R})$. Then, there exixts an entire function  $g$ such that $\left| f(x) - g(x) \right| < \epsilon(x), x \in \mathbb{R}$.  
\end{theorem}

Note that $f$ can be approximated much better than uniformly, since $\epsilon(x)$ may decrease to zero with arbitrary speed, as $x\rightarrow\infty.$ 
Of course, since every continuous function on a bounded closed interval extends continuously to $\R$ and since entire functions are represented by their Maclaurin series, the Weierstrass theorem is contained in the Carleman theorem. There are many proofs of the Weierstrass theorem in various textbooks on approximation, but the original proof of Weierstrass actually used a preliminary version of the Carleman theorem. That is, Weierstrass began by approximating a continuous function on a bounded closed interval by entire functions and then approximating the entire function by partial sums of its Maclaurin series. 

Let $I = (a, b)$ be an interval, $-\infty \leq a \leq b \leq +\infty$. Denote by $I^c = \mathbb{R}\backslash I$ the complement of $I$ in $\R$. 
For an open subset $U \subset\mathbb{C}$, we denote by $H(U)$ the family of functions holomorphic on $U$.

\begin{theorem} Let $m$ be a non-negative integer, $I = (a, b)$, $f \in C^m(I),$ and $\epsilon \in C^+(I).$ Then, there exists a function  $g\in H(\mathbb{C} \backslash I^c),$ such that $|f^{(i)}(x) - g^{(i)}(x)| < \epsilon(x), x \in I, i=0,1,\cdots,m.$
\end{theorem}

In particular, we have the following generalization of Theorem 1, due to Hoischen \cite[Satz 3]{H}. 

\begin{corollary} Let $m$ be a non-negative integer,  $f \in C^m(\mathbb{R}),$ and $\epsilon \in C^+(\mathbb{R}).$ Then, there exists an entire function $ g,$ such that $|f^{(i)}(x) - g^{(i)}(x)| < \epsilon(x), x \in \mathbb{R}, i=0,1\cdots, m.$
\end{corollary} 

The results of Carleman and  Hoischen have been extended in various directions. For example, Carleman's theorem was extended by Scheinberg to approximation by entire functions of several complex variables and Frih and Gauthier \cite[Corollary]{FG} showed the corresponding extension of the theorem of Hoischen on the simultaneous approximation of derivatives. In \cite{S} and \cite{FG} the functions to be approximated are defined on  the real part $\R^N$ of $\C^N=\R^N+i\R^N$ and they are approximated by functions holomorphic in all of $\C^N.$  Very recently, Johanis \cite{J} has considered the more general problem of approximating a function $f$ given on only a portion $\Omega$ of $\R^N.$ Whitney's famous theorem \cite{W} allows one to approximate such functions $f$ by functions analytic on $\Omega.$  Of course every  function  analytic on $\Omega$ naturally extends holomorphically to a neighborhood of $\Omega$ in $\C^N,$ but this neighborhood will depend on the analytic function.  
The beautiful result of Johannis shows that there is a domain $\widetilde\Omega\subset\C^N,$  depending only on $\Omega$ and not on $f,$ such that $f$ can be approximated by functions holomorphic on $\widetilde\Omega.$ When applied to our situation, where $N=1$ and $\Omega$ is an interval $I,$ the domain $\widetilde I$ is smaller than the domain $\C\setminus I^c$ which we obtain in Theorem 3.

For a closed set $E\subset\C,$ let $A(E) \equiv C(E) \cap H(E^o)$. 
In the Carleman theorem, if we replace the real line $\R$ by a a closed subset $E\subset\C,$ then the function $f$ to be approximated must be, not only continuous on $E,$ but also holomorphic on the interior of $E.$ That is, $f$ must lie in $A(E).$
A condition on sets $E,$ necessary for the possibility of such approximations, was introduced in \cite{Gau}, and in \cite{N} this condition was shown to be also sufficient.  

The techniques employed in previous papers are quite technical. The aim of the present note is to show that Theorem 3, extending Theorem 1 to open intervals and $\epsilon$ decreasing to zero with arbitrary speed, can be proved in the same way as the elementary proof of Theorem 1, that is, by approximating the derivative of highest order and integrating. 


\section{Preliminaries}

A fundamental lemma, known as the Walsh Lemma, asserts that, for a compact set $K\subset\C,$ if every function $f\in A(K)$ can be uniformly approximated by rational functions having no poles on $K,$ then, not only are there rational functions which uniformly approximate $f,$ there are even rational functions which, in addition to approximating $f,$ also simultaneously interpolate $f$ at finitely many given points of $K.$ This Walsh Lemma has been extended to the context of functional analysis. 
For a topological vector space $X$, we denote by $X^*$ the (continuous) dual. If $X$ and $Y$ are topological vector spaces, where $X$ is a subspace of $Y$, then of course $Y^* \subset X^*$. 
The following result on simultaneous approximation and interpolation is a generalization of  the Walsh  Lemma due to Deutsch \cite{D}. 

\begin{lemma} Let $X$ be a dense subspace of a normed vector space $Y$. Let $y \in Y$, $\epsilon$ be a positive number and $ L_1, \cdots , L_n \in Y^*.$ Then, there exists $x \in X$ such that $\left| y - x \right| < \epsilon$ and $ L_i(y) = L_i(x), \, i = 1, \cdots , n.$
\end{lemma}

A compact set $K\subset\C$ is said to be a set of polynomial approximation, if for each $f\in A(K)$ and $\epsilon>0,$  there exists a polynomial $p$  such that $|f - p| < \epsilon$. The celebrated theorem of Mergelyan (see \cite{Gai}) states that  a compact set is a set of polynomial approximation if and only if its complement is connected.  A particular case of the Walsh lemma is the following. 

\begin{lemma} Let $K \subset \mathbb{C}$ be a compact set of polynomial approximation. Then, for all $\phi \in A(K)\, , L_i \in A(K)^*\, , i = 1, \cdots ,n$ and for all $\epsilon> 0,$ there exists a polynomial $p,$   such that $|\phi - p| < \epsilon$ and $ L_i(\phi) = L_i(\psi), \, i = 1, \cdots , n.$
\end{lemma}

For $\phi\in C([j,j-1]), j \in \mathbb{Z},$  set
$$\begin{array}l
{T_1}^{j}(\phi) = \int_{j - 1}^{j}\int_{0}^{x_1}\int_{0}^{x_2} \cdots \int_{0}^{x_{m - 1}}\phi (t) dt dx_{m - 1} \cdots dx_1,\\
\\
{T_2}^{j}(\phi) = \int_{j - 1}^{j}\int_{0}^{x_1}\int_{0}^{x_2} \cdots \int_{0}^{x_{m - 2}}\phi (t) dt  dx_{m - 2} \cdots dx_1,\\
.\\
.\\
.\\
{T_m}^{j}(\phi) = \int_{j - 1}^{j}\phi (t) dt.\\
\end{array}$$

\begin{lemma}
For all $f \in C(\mathbb{R})$ and for all $\epsilon \in C^+(\mathbb{R}),$ there exists an entire function $g$ such that $g(0)=f(0); \, {T_i}^{j}(g) = {T_i}^{j}(f)$ for $i = 1, 2, \cdots , m$ and  $j \in \mathbb{Z};$ and $|f(t) - g(t)| < \epsilon(t)$.
\end{lemma}

\begin{proof} 
First of all, we can see that Lemma 7 is true for a finite number of $j,$ by applying Lemma 6 to a closed interval $E$ containing the intervals $[j-1,j]$ in question.   

Let $f \in C(\mathbb{R})$ and $\epsilon\in C^+(\mathbb{R}).$ 
We may assume that $\epsilon(t)=\epsilon(|t|)$ and $\epsilon(|t|)$ is decreasing as $|t|$ grows. 
Let $\{\epsilon_k\}$ be a sequence of positive numbers such that $\epsilon_k < \epsilon(k)$ and $\sum_{k = \ell}^{\infty}\epsilon_k < \epsilon(t)/2$ for $\ell \geq |t|$, $t \in \mathbb{R}$. We may choose, $\epsilon_k = \epsilon(k)/2^{k+2}$.
Indeed, 
\[\sum_{k = \ell}^{\infty}\epsilon_k = \sum_{k = \ell}^{\infty}\epsilon(k)/2^{k+2} \leq \epsilon(\ell)\sum_{k = \ell}^{\infty}1/2^{k+2} = \epsilon(\ell)/2^{\ell+1} \le \epsilon(t)/2.\]
Now, for each $k\in\N,$ set  
$E_k = \overline{D_{k - 1}} \, \bigcup \, [-k, -(k - 1)] \, \bigcup \, [k - 1, k],$ where $D_r$ is the disc of center $0$ and radius $r.$  

By Lemma 6 and the Weierstrass approximation theorem, there exists a polynomial $g_1$ such that $|f - g_1| < \epsilon_1$ on $[-1 , 1], f(j) = g_1(j),$ for  $j = - 1, 0, 1$  and 
$$
	T^j_i(f)=T^j_i(g_1), \, i=1,\cdots,m; \quad  j=0, 1. 
$$

Set  
\[h_2 = 
\left \{ 
\begin{array}{ll}
g_1 \mbox{ on } \overline{D_1}\\
f \mbox{ on } [-2, 2] \, \backslash \,[-1, 1]. \\
\end{array} \right. \]
Since $h_2\in A(E_2),$ it follows from Lemma 6 and the Mergelyan theorem that there is  a polynomial $g_2$ such that $|h_2 - g_2| < \epsilon_2$ on $E_2$, $h_2(j) = g_2(j)$ for $j = -2, -1, \cdots, 2$ and such that 
$$
T^j_i(h_2)=T^j_i(g_2), \, i=1,\cdots,m; \quad  j= -1, 0, \cdots, 2. 
$$ 
Thus, we have 
$$
T^j_i(f)=T^j_i(g_2), \, i=1,\cdots,m;  \quad  j= -1, 0, \cdots, 2; 
$$
$$
 f(j) = g_2(j) \mbox{ for } j = -2, -1, \cdots, 2;
$$
and
\[|f - g_2| <  
\left \{ 
\begin{array}{ll}
\epsilon_2 \mbox{ on } [-2, 2] \,\backslash \,[-1, 1]\\
\epsilon_1 + \epsilon_2 \mbox{ on } [-1, 1]. \\
\end{array} \right. \]

Indeed, \[|f - g_2| \leq |f - h_2| + |h_2 - g_2| = \]

\[\left \{ 
\begin{array}{ll}
0 + |h_2 - g_2| < \epsilon_2 \mbox{ on } [-2, 2] \, \backslash \, [-1, 1]\\
|f - g_1| + |h_2 - g_2| < \epsilon_1 + \epsilon_2 \mbox{ on } [-1, 1] .\\
\end{array} \right. \]
We also have that $|g_2 - g_1| < \epsilon_2$ on $\overline D_1.$ Indeed, 
\[|g_2 - g_1| \leq |g_2 - h_2| + |h_2 - g_1| < \epsilon_2 + 0 \mbox{ on } \overline D_1.\]

Setting $g_o=g_1,$ we shall show by induction that for $k= 1, 2, \cdots, $ there exist polynomials  $g_k,$ such that 
\begin{equation}
	T^j_i(f)=T^j_i(g_k), \, i=1,\cdots,m;  \quad  j= -(k - 1), \cdots, k .
\end{equation}
\begin{equation}
     f(j) = g_k(j) \mbox{ for } j = -k,\cdots, k
\end{equation}
\begin{equation}
|f - g_k| <
\left \{ 
\begin{array}{ll}
\epsilon_k \mbox{ on } [-k, k] \,\backslash \,[-(k - 1), k - 1]\\
\epsilon_{k - 1} + \epsilon_k \mbox{ on }  [-(k - 1), k - 1] \,\backslash \,[-(k - 2), k - 2]\\
.\\
.\\
.\\
\epsilon_1 + \epsilon_2 + \cdots + \epsilon_k \mbox{ on } [-1, 1]\\
\end{array} \right.
\end{equation}
and
\begin{equation}
|g_k - g_{k - 1}| < \epsilon_k \mbox{ on } \overline D_{k - 1}.
\end{equation}

As shown before, we already verified the cases $k = 1$ and $2$. We suppose the validity of the cases $k = 1, \cdots, n$. Set
\[h_{n + 1} = 
\left \{ 
\begin{array}{ll}
g_{n} \mbox{ on } \overline{D_{n}}\\
f \mbox{ on } [-(n + 1) , n + 1 ] \, \backslash \,[-n, n] .\\
\end{array} \right. \]
There exists a polynomial $g_{n + 1}$ such that $|h_{n + 1} - g_{n + 1}| < \epsilon_{n + 1}$ on $E_{n + 1}$, $h_{n + 1}(j) = g_{n + 1}(j)$ for $j = -(n + 1),\cdots, n + 1$ and such that 
$$
	T^j_i(h_{n + 1})=T^j_i(g_{n + 1}), \, i=1,\cdots,m; \quad  j= -n,\cdots, n + 1. 
$$ 
Thus, we have 
$$
T^j_i(f)=T^j_i(g_{n + 1}), \, i=1,\cdots,m;  \quad   j= -n,\cdots, n + 1,
$$
$$
 f(j) = g_{n + 1}(j) \mbox{ for } j = -(n + 1),\cdots, n + 1.
$$
and 
\[|f - g_{n + 1}| <
\left \{ 
\begin{array}{ll}
\epsilon_{n + 1} \mbox{ on } [-(n + 1), n + 1] \,\backslash \,[-n, n ]\\
\epsilon_{n} + \epsilon_{n + 1} \mbox{ on }  [-n, n] \,\backslash \,[-(n - 1), n - 1]\\
.\\
.\\
.\\
\epsilon_1 + \epsilon_2 + \cdots + \epsilon_{n + 1} \mbox{ on } [-1, 1]\\
\end{array} \right.\]
since
\[|f - g_{n + 1}|  \leq |f - h_{n + 1}| + |h_{n + 1} - g_{n + 1}| = \]
\[\left \{ 
\begin{array}{ll}
0 + |h_{n + 1} - g_{n + 1}| < \epsilon_{n + 1} \mbox{ on } [-(n +1), n + 1] \, \backslash \, [-n, n]\\
|f - g_n| + |h_{n + 1} - g_{n + 1}| < \epsilon_n + \epsilon_{n + 1} \mbox{ on } [-n, n] \, \backslash \, [-(n - 1), n -1]\\
.\\
.\\
.\\
|f - g_n| + |h_{n + 1} - g_{n + 1}| < \epsilon_1 + \epsilon_2 + \cdots + \epsilon_{n + 1} \mbox{ on } [-1, 1]. \\
\end{array} \right. \]
We also have that $|g_{n+1} - g_n| < \epsilon_{n + 1}$ on $\overline D_n,$ since
\[|g_{n + 1} - g_n| \leq |g_{n + 1} - h_{n + 1}| + |h_{n + 1} - g_n| < \epsilon_{n + 1} \mbox{ on } \overline D_n.\] 

Let us show that the sequence $\{g_k\}$ converges uniformly on compacta. 
It is sufficient to show that $\lbrace g_k \rbrace$, is uniformly Cauchy on compact subsets.
For each $k$, we have that $|g_k - g_{k - 1}| < \epsilon_k$ on $\overline{D_{k - 1}}$. Let $K \subset \mathbb{C}$ be an arbitrary compact set.
For $\delta > 0$, we choose $N_{\delta}$ so large that $K \subset D_{N_{\delta}}$ and $k > \ell > N_{\delta} \Rightarrow \sum^{k}_{j = \ell} \epsilon_j < \delta$.
Then, for such $k$ and $\ell$,\[|g_k - g_\ell| \leq \sum_{j = \ell}^{k - 1}|g_{j + 1} - g_j| \leq \sum_{j = \ell}^{k - 1}\epsilon_{j + 1} < \delta \mbox{, $\quad$ on } K.\] Thus, the sequence $g_k$ converges uniformly on compacta. The limit $g$ is therefore an entire function.

Let us show that $|f - g| < \epsilon$. Fix $t \in \mathbb{R}$ and choose $\ell = [|t|] + 1$. Then, for all $k \geq \ell,$
$$ 	
	|f(t) - g_k(t)| \le |f(t)-g_m(t)|+\sum_{j=\ell+1}^k|g_j(t)-g_{j-1}(t)|<
 \sum_{j = \ell}^{k}\epsilon_j < \epsilon(t)/2.
$$
Now, we choose $k \geq \ell$ so large that $|g_k(t) - g(t)| < \epsilon(t)/2$. Then,
\[|f(t) - g(t)| \leq |f(t) - g_k(t)| + |g(t) - g_k(t)| < \epsilon(t).\]

Finally, we must show that $T^j_i(g)=T^j_i(f), \, i=1,\cdots,m \,;  \quad  j= \Z$. Fix $j.$ For all $k>|j|,$ we have $j\in\{(k-1),\cdots,k\}.$ Thus, by (1), 
$$
	T^j_i(f)=T^j_i(g_k), \, i=1,\cdots,m
$$
and consequently, 
$$
	T^j_i(g)=\lim_{k\rightarrow\infty}T^j_i(g_k)=\lim_{k\rightarrow\infty}T^j_i(f)=T^j_i(f).
$$
\end{proof}


\section{Proof of Theorem 3}

\begin{proof}:  For simplicity, we shall prove Corollary 4, which is a special  case of Theorem 3.  The proof of the general theorem is an obvious modification. 

We may assume that $f^{(i)}(0)=0, i=0,\cdots,m$ and we may also assume that $\epsilon (t)=\epsilon(|t|)$ and that $\epsilon(|t|)$ is decreasing, as $|t|$ increases. Put $\epsilon_o=\epsilon,$ and for $i=1,\cdots,m;$ put $\epsilon_i(t)=\epsilon_{i-1}(|t|+1).$ Then, for $i=0,\cdots,m,$ we have $\epsilon_i(t)=\epsilon(|t|),$ the functions $\epsilon(t)$ are decreasing as $|t|$ increases and $\epsilon_i>\epsilon_{i+1}, i=0,\cdots,m-1.$ 
  
By Lemma 7 there exists an entire function  $g_m$ such that $g_m(0)= f^{(m)}(0)$;

 \[\left| f^{(m)}(t) - g_m(t) \right| < \epsilon_m(t);\] 
\[\int_{n - 1}^{n} f^{(m)}(t) dt = \int_{n - 1}^{n} g_m(t) dt;\]
and
$$
\int_{n - 1}^{n}\int_{0}^{x_1} \cdots \int_{0}^{x_{i}} f^{(m)}(t) dt dx_{i} \cdots  dx_1 =
$$
$$
\int_{n - 1}^{n}\int_{0}^{x_1} \cdots \int_{0}^{x_{i}} g_m(t)  dt dx_{i} \cdots  dx_1;
$$
$$
	\mbox{for} \quad \quad i=1,\cdots,m-1; \quad \mbox{and} \quad n\in\Z.
$$
We define the following entire functions.

\[g_k(z) = \int_{0}^{z} g_{k+1}(\zeta) d\zeta; \quad k=m-1, m-2, \cdots,0.\]
Thus, we have:
$$
\begin{array}{llll}
	g_k^\prime(z)	&	=	&	g_{k+1}(z)	,				&	k=0,\cdots,m-1	\\
				&	=	&	\int_0^zg_{k+2}(\zeta)d\zeta,	&	k=0,\cdots,m-2.		
\end{array}
$$
Hence, setting $g=g_0,$ we have:
$$ 
	g'(z) = g_1(z), \quad g''(z) = g_2(z), \quad  \cdots \quad  g^{(m)}(z) = g_m(z).
$$
Therefore
$$
	 \left| f^{(m)}(x) - g^{(m)}(x) \right| =  \left| f^{(m)}(x) - g_m(x) \right| < \epsilon_m(x) \le \epsilon(x).
$$
We shall now show that
\[\left| f^{(m - 1)}(x) - g^{(m - 1)}(x) \right| < \epsilon_{m-1}(x) \le\epsilon(x).\]
Denoting the integer part of $x$ by $[x],$ we have, if $x\ge 0:$
\[ 
	\left| f^{(m - 1)}(x) - g^{(m - 1)}(x) \right| = \left| \int_{0}^{x} \lbrack f^{(m)}(t) - g^{(m)}(t) 	\rbrack dt \right| =
\]
\[
	\left| \sum_{n = 1}^{[x]} \int_{n - 1}^{n} \lbrack f^{(m)}(t) - g^{(m)}(t) \rbrack dt + 				\int_{[x]}^{x} \lbrack f^{(m)}(t) - g^{(m)}(t) \rbrack dt \right| =
\] 
\[
	\left| \int_{[x]}^{x} \lbrack f^{(m)}(t) - g^{(m)}(t) \rbrack dt \right| \leq 
	 \epsilon_m([x]) = 
\]
\[
	\epsilon_{m - 1}([x] +1) < \epsilon_{m - 1}(x) \le \epsilon(x).
\]
Similarly, if $x\le 0,$
\[ 
	\left| f^{(m - 1)}(x) - g^{(m - 1)}(x) \right| = \left| \int_{0}^{x} \lbrack f^{(m)}(t) - g^{(m)}(t) 	\rbrack dt \right| =
\]
\[
	\left| \sum_{n = -1}^{[x]} \int_{n+1}^{n} \lbrack f^{(m)}(t) - g^{(m)}(t) \rbrack dt + 				\int_{[x]}^{x} \lbrack f^{(m)}(t) - g^{(m)}(t) \rbrack dt \right| =
\] 
\[
	\left| \int_{[x]}^{x} \lbrack f^{(m)}(t) - g^{(m)}(t) \rbrack dt \right| \leq 
	 \epsilon_m([x]) = 
\]
\[
	\epsilon_{m - 1}([x] +1) < \epsilon_{m - 1}(x) \le \epsilon(x).
\]

Next we show that $|f^{(m - 2)}(x) - g^{(m - 2)}(x)| < \epsilon_{m - 2}(x) \le \epsilon(x).$ As in the previous case, 
\[\left| 
	f^{(m - 2)}(x) - g^{(m - 2)}(x) \right| = 
	\left| \int_{0}^{x}\lbrack f^{(m - 1)}(x_1) - g^{(m - 1)}(x_1)\rbrack dx_1  \right| =
\]
\[
	\left| \int_{[x]}^{x}\lbrack f^{(m - 1)}(x_1) - g^{(m - 1)}(x_1) \rbrack dx_1 \right|  \leq
\]
\[
	\epsilon_{m - 1}([x]) = \epsilon_{m - 2}([x] +1) < \epsilon_{m - 2}(x) \le \epsilon(x).
\]
Repeating the same argument $m - 2$ times, we obtain that $|f^{(i)}(x) - g^{(i)}(x)| < \epsilon(x), x\in \mathbb{R}$ and $i = 0, 1, \cdots , m$.

\end{proof}


\end{document}